\documentclass[a4paper]{amsart}
\usepackage{a4,amssymb,mathrsfs,pstricks,color,cancel}

\newcommand{\R}{\mathbb{R}}
\newcommand{\C}{\mathbb{C}}
\newcommand{\N}{\mathbb{N}}
\newcommand{\D}{\mathbb{D}}
\newcommand{\T}{\mathbb{T}}

\newcommand{\A}{\mathscr{A}}
\newcommand{\supp}{\text{supp}}
\newcommand{\Res}{\text{Res}}
\newtheorem{thm}{Theorem}
\newtheorem{lem}[thm]{Lemma}
\newtheorem{cor}[thm]{Corollary}
\newtheorem{prop}[thm]{Proposition}

\theoremstyle{definition}

\newtheorem{eg}{Example}

\newtheorem{prop2}[thm]{Proposition}
\newcommand{\vp}{\varphi}
\newcommand{\ve}{\varepsilon}

\title[Adjoints of Composition Operators on $\C^+$]{Adjoints of composition operators on Hardy spaces of the half-plane}
\author{Sam Elliott}
\address{Department of Pure Mathematics\\ University of Leeds\\ Leeds\\ LS2 9JT\\ UK}
\email{samuel@maths.leeds.ac.uk}
\begin{document}
\maketitle
\begin{abstract}
Building on techniques used in the case of the disc, we use a variety of methods to develop formulae for the adjoints of composition operators on Hardy spaces of the upper half-plane. In doing so, we prove a slight extension of a known necessary condition for the boundedness of such operators, and use it to provide a complete classification of the bounded composition operators with rational symbol. We then consider some specific examples, comparing our formulae with each other, and with other easily deduced formulae for simple cases.
\end{abstract}

\section*{Introduction}
A great deal of work has already taken place in studying the properties of analytic composition operators on Hardy spaces on the unit disc $\D$ of the the complex plane. It has long been known that all such operators are bounded on all the Hardy spaces (and indeed on a great many other spaces too), and a number of characterisations of compactness and weak compactness have also been produced, including those of Cima and Matheson \cite{Cima2}, Sarason \cite{Sarason} and Shapiro \cite{Shapiro}.

In contrast, relatively little is known about composition operators acting on Hardy spaces of a half-plane. Although corresponding Hardy spaces of the disc and half-plane are isomorphic, composition operators act very differently in the two cases. It is known, for example, that not all analytic composition operators are bounded, though no satisfactory characterisation of boundedness has yet been found; moreover, Valentin Matache showed in \cite{Matache} that there are in fact no compact composition operators in the half-plane case. The question of when an operator is isometric has now been dealt with in both cases, however: in the disc by Nordgren \cite{Nordgren}, and more recently in the half-plane by Chalendar and Partington \cite{Chalendar03}.

Lately, a good deal of research has concentrated on describing the adjoints of analytic composition operators on the disc. Much of this work has been concerned with the Hardy space $H^2(\D)$ which, being a subspace of $L^2(\T)$, is a Hilbert space and hence self-dual, meaning adjoints play a particularly important r\^ole in its structure. Here $\T$ denotes the unit circle in the complex plane.

In \cite{Cowen2} Carl Cowen produced the first adjoint formulae for the case where the composing map is fractional linear. It has since been shown that for all the Hardy spaces on the disc, an important generalisation of Aleksandrov's Disintegration Theorem \cite{Aleksandrov87}, gives rise to a formula for for the pre-adjoint of a composition operator in terms of what are now called Aleksandrov-Clark (AC) measures. The same method has been shown to work for the $L^p$ spaces on $\T$, and even the space of Borel measures on $\T$. In particular, since $H^2$ and $L^2$ are both Hilbert spaces, in these cases this formula gives a description of the adjoint of the composition operator as well.

More recently, John McDonald \cite{McDonald} produced an explicit adjoint formula for operators induced by a finite Blaschke products. In the last few years, Cowen together with Eva Gallardo-Guti\'errez \cite{Cowen1} developed a method, later corrected by Hammond, Moorehouse and Robbins \cite{Hammond}, which gave a characterisation in the more general case of an operator on $H^2$ with rational symbol. The formula shows that the adjoint of each such operator is a so-called `multiple-valued weighted composition operator', plus an additional term. A simplified proof of the formula has since been given by Paul Bourdon and Joel Shapiro \cite{Bourdon}.

We begin by generalising the notion of Aleksandrov-Clark measures to the half-plane (we choose the upper half-plane $\C^+$, as its boundary is the most natural to work with for our purposes). This generalisation has already been made by a number of authors, though not with our intentions in mind. We show that, subject to a certain condition necessary for a composition operator to be bounded, a characterisation of the pre-adjoint of a composition operator can also be made on the half-plane using AC measures; again this will give an adjoint formula for the case where $p=2$.

The middle sections of this paper will then be devoted to the study of composition operators with rational symbol. We prove a complete characterisation of the boundedness of such operators, as well as a number of other results along the same lines. Having made this characterisation, we use integral methods in the vein of \cite{Hammond} to find an explicit formula for the adjoint of a composition operator on $H^2(\C^+)$ with rational symbol, which will turn out to be a multiple-valued weighted composition operator, but this time without any additional terms.

Finally, we present some examples including the simplest case (an operator with linear symbol), and a slightly more complicated function known to be an isometry by the results of \cite{Chalendar03}.

\section{Preliminaries}

For $1\le p<\infty$, the Hardy space $H^p(\C^+)$ is the Banach space of analytic functions $f:\C^+\rightarrow\C$ such that the norm
\[
 \|f\|_p = \sup_{y\in\R} \left(\int_\R |f(x+iy)|^p dx\right)^\frac{1}{p} < \infty \text{.}
\]
The space $H^\infty(\C^+)$ is the space of all bounded analytic functions on $\C^+$ together with the supremum norm. It can easily be shown that each $H^p$-space is a subspace of the corresponding $L^p(\R)$-space by equating each Hardy space function with its boundary function, reached via non-tangential limits; equivalently, for $p<\infty$ it is possible to extend any $L^p$ function to the half-plane by integrating with respect to the Poisson kernels. As such, we see that $H^2$ is in fact a Hilbert space, being a subspace of $L^2$. An analagous construction may be made for the disc, and a natural identification of the disc with half-plane induces an isomorphism between each $H^p(\C^+)$ and the equivalent Hardy space of the disc. We will explore this identification further in Section \ref{Rational}.

For an analytic map $\vp:\C^+\rightarrow\C^+$, we may define the composition operator with symbol $\vp$, which can be considered to act on any of the spaces $H^p(\C^+)$ or $L^p(\R)$. Given such a mapping $\vp$, this operator, written $C_\vp$ is defined by the formula
\[
 C_\vp f = f\circ \vp \text{.}
\]
For $f\in L^p$, we may either extend $f$ to the half-plane and compose it with $\vp$, or extend $\vp$ to $\R$ and use this for composition, the two methods are entirely equivalent. In the case of disc, the Aleksandrov-Clark (AC) measures of an analytic function, $\psi : \D\rightarrow\D$, were constructed via the collection of functions given by
\[
 u_\beta (z) = \Re \left( \frac{\beta + \psi (z)}{\beta - \psi (z)} \right) \text{,}
\]
for $\beta\in\T$. Each $u_\beta$ can be shown to be positive and harmonic on the disc, and so, via the Riesz-Herglotz representation theorem, each may be written as the Poisson integral of a (finite) positive measure on the unit circle. This collection, indexed by $\T$, is known as the collection of Aleksandrov-Clark (AC) measures associated with $\psi$, and denoted $\mathscr{A}_\psi$. For a full description of the construction, see for example \cite{Cima,Kapustin06,Saksman}.

A number of results are well known about AC measures in the disc case, most particularly the following theorem from \cite{Aleksandrov87}, reproduced in a number of other works, including for example \cite{Cima}, page 216.

\begin{thm}\textbf{(Aleksandrov's Disintegration Theorem)}\label{Disintegration}

Let $\psi$ be an analytic self-map of the disc, and $\mathscr{A}_\psi=\{\mu_\beta:\beta\in\T\}$ be the collection of AC measures associated with $\psi$. Then for each function $f\in L^1(\T)$,
\[
\int_\T\left(\int_\T f(\zeta) d\mu_\beta(\zeta)\right) dm(\beta) = \int_\T f(\zeta) dm(\zeta) \text{,}
\]
where $m$ denotes normalised Lebesgue measure on $\T$.
\end{thm}

\bigskip

In the upper half-plane case, the equivalent construction is as follows:
given an analytic self-map of the upper half-plane, $\varphi$, we note that the function
\[
 u_\alpha (z) = \Re \left(\frac{i (1 + \alpha\varphi(z))}{\varphi(z) - \alpha} \right)
\]
is positive, and harmonic for each $\alpha\in\R$. In fact, this is precisely the function we get by transforming the plane to the disc via the standard M\"obius identification:
\[
 \begin{array}{rlllcrl}
  J:&\D&\rightarrow&\C^+      & & z &\mapsto  i \left(\frac{1-z}{1+z}\right) \\
  J^{-1}:&\C^+&\rightarrow&\D & & s &\mapsto \frac{ i -s}{ i +s} \text{,}
 \end{array}
\]
taking the function $u_\alpha$ from the disc case, and transforming back to the plane. Since $\alpha$ is simply a constant with respect to $z$, the functions given by
\[
 v_\alpha(z) = \frac{1}{\pi(1+\alpha^2)}\Re \left(\frac{i (1 + \alpha\varphi(z))}{\varphi(z) - \alpha} \right)
\]
are also all positive and harmonic, and we will see later that it will be more convenient to use this system for our purposes. We continue with the following theorem from \cite{Garnett}.

\begin{thm}\textbf{(The Upper Half-Plane Herglotz Theorem)}\label{Herglotz}

We denote by $P_y(x-t)$ the upper half-plane Poisson kernel, namely
\[
 P_y(x-t) = \frac{1}{\pi}\frac{y}{(x-t)^2+y^2}\text{.}
\]
If $v:\C^+\rightarrow\R$ is a positive, harmonic function, then $v$ may be written as
\[
 v(x+i y) = c y + \int_\R P_y(x-t) d\mu(t) \text{,}
\]
where $c\ge0$ and $\mu$ is a positive measure such that
\[
 \int\frac{d\mu(t)}{1+t^2} < \infty \text{.}
\]
\end{thm}

We notice that, unlike the disc case, in the half-plane we lose the finiteness of our measures, and there is an additional term of $c\Im(z)$. This additional term corresponds to a point mass existing at a notional point `$\infty$', or equivalently to a point mass at $-1$ on the boundary of the disc.

Using Theorem \ref{Herglotz}, we see that each $v_\alpha$ may be written
\begin{equation}\label{valpha}
 v_\alpha (x+i y) = c_\alpha y + \int_\R P_y(x-t) d\mu_\alpha(t) \text{.}
\end{equation}

We will call the collection of pairs $(\mu_\alpha,c_\alpha)$ the Aleksandrov-Clark (AC) measures associated with $\varphi$, and much as in the disc, denote this collection $\A_\vp$.

\section{The half-plane Aleksandrov Operator}

We begin by noting the following result, which will simplify our future calculations.

\begin{lem}\label{calphazero}
For any function $\varphi:\C^+\rightarrow\C^+$, the constant $c_\alpha$ in \eqref{valpha} takes the value zero for $m$-almost every $\alpha$.
\end{lem}
\begin{proof}
We take the function $\varphi$, and construct the collection of functions $v_\alpha$ as above. For $\alpha\in\R$, we denote by $\hat{\alpha}$ the corresponding point on the circle $\T$, via the standard identification.

We may also translate the functions $\vp$ and $v_\alpha$ to eqivalent functions on the disc: we denote
\[
 \begin{array}{rlllcrl}
  \widetilde{\vp}:&\D&\rightarrow&\D      & & \widetilde{\vp} & = J^{-1} \circ \vp \circ J \\
  \widetilde{v}_{\alpha}:&\D & \rightarrow & \R^+ & & \widetilde{v}_{\alpha} & = v_\alpha \circ J \text{.}
 \end{array}
\]
We observe  that
\begin{align*}
 \widetilde{v}_{\alpha} (z) & = \frac{1}{\pi(1+\alpha^2)}\Re \left(\frac{i (1 + \alpha\varphi(J(z)))}{\varphi(J(z)) - \alpha} \right) \\
  & = \frac{1}{\pi(1+\alpha^2)}\Re \left(\frac{\hat{\alpha}+\widetilde{\vp}(z)}{\hat{\alpha}-\widetilde{\vp}(z)} \right)\text{,}
\end{align*}
by construction. So the functions $\widetilde{v}_{\alpha}$ are positive multiples of the functions $u_{\hat{\alpha}}$, and the measures they define via Herglotz' Theorem will have point masses in the same places.

By Garnett (\cite{Garnett}, page 19), the value of $c_\alpha$ corresponds to the point mass of the measure given by $\widetilde{v}_{\alpha}$ at $-1$. As such, $c_\alpha$ is zero if and only if that measure has no point mass at $-1$, or equivalently by the above, the AC measure associated with $\widetilde{\varphi}$ and $\hat{\alpha}$ has no point mass at $-1$.

Let us suppose that the AC measure associated with $\widetilde{\varphi}$ and $\hat{\alpha}$ had a non-zero point mass at ${-1}$ for a set of $\hat{\alpha}$ of positive Lebesgue measure. We denote by $\{\mu_\alpha\}$ the collection of all AC measures associated with $\widetilde{\varphi}$.

Let $f$ be an $L^1$ function on $\T$, then by the standard Aleksandrov disintegration theorem (Theorem \ref{Disintegration} above) we have
\begin{align*}
 \int_\T f(\zeta) dm(\zeta) & = \int_\T \int_\T f(\zeta) d\mu_\alpha(\zeta) dm(\alpha) \\
  & = \int_\T \int_{\T\setminus\{-1\}} f(\zeta) d\mu_\alpha(\zeta) dm(\alpha) + \int_\T \int_{\{-1\}} f(\zeta) d\mu_\alpha(\zeta) dm(\alpha) \\
  & = \int_\T \int_{\T\setminus\{-1\}} f(\zeta) d\mu_\alpha(\zeta) dm(\alpha) + f(-1) \int_\T k_\alpha dm(\alpha) \text{,}
\end{align*}
where $k_\alpha$ is the value of the point mass of $\mu_\alpha$ at $-1$. If we change the value of $f$ at the point $-1$ (a set of Lebesgue measure zero), the left hand side of this equality will remain unchanged, but the right hand side will change, since $k_\alpha$ is non-zero on a set of positive Lebesgue measure. This is a contradiction, hence $\mu_\alpha(\{-1\})$ cannot be non-zero on a set of positive Lebesgue measure, and thus $c_\alpha = 0$ for $m$-almost every $\alpha$.
\end{proof}

We now define the Aleksandrov Operator, $A_\varphi$, of symbol $\varphi$, to be the operator
\[
 A_\varphi f(\alpha) = \int_\R f(t) d\mu_\alpha(t) \text{.}
\]
This operator may be allowed to act on any number of function spaces on the upper-half plane, but for the moment, we will simply consider this definition to be true `whenever the integral makes sense'. It is clear that this is a linear operator.

\subsection{Functions which map $\infty$ to itself} We begin by looking at how the Aleksandrov Operator acts on a Poisson kernel. Taking $z=x+i y$, we let $f_z(t) = P_y(x-t)$. By definition, we have
\begin{align*}
 A_\varphi f_z(\alpha) & = \int_\R P_y(x-t) d\mu_\alpha(t) \\
& = \frac{1}{\pi(1+\alpha^2)} \Re \left( i \frac{1 + \alpha \varphi(z)}{\varphi(z) - \alpha} \right) \qquad \qquad \tag{$\begin{smallmatrix} \text{$m$-almost everywhere,} \\ \text{by Lemma \ref{calphazero}} \end{smallmatrix}$}\\
& = \Re \left( \frac{i}{\pi(1 + \alpha^2)} \frac{\overline{\varphi(z)} - \cancelto{\text{imaginary}}{\alpha + \alpha|\varphi(z)|^2} - \alpha^2\varphi(z)}{(\Re(\varphi(z)) -\alpha)^2 + \Im(\varphi(z))^2} \right)
\intertext{Rearranging, we get}
 & = \frac{1}{\pi(\cancel{1+\alpha^2})} \frac{\cancel{(1+\alpha^2)} \Im (\varphi(z))}{(\Re(\varphi(z)) -\alpha)^2 + \Im(\varphi(z))^2} \\
& = P_{\Im (\varphi(z))} (\Re(\varphi(z)) - \alpha) \\
& = f_{\varphi(z)} (\alpha) \text{.}
\end{align*}
We aim to show a level of duality between the Aleksandrov Operator, and the composition operator $C_\varphi$. Let us for the moment assume that
\[
 \varphi(\infty) = \lim_{|z|\rightarrow\infty} \varphi(z) = \infty \text{.}
\]
Then for each $M\in\N$, there is some $N\in\N$ such that $|\varphi(z)|>M$ whenever $|z|>N$. In particular, if $g$ has compact support in $\R$, then
\[
 \supp(g)\subseteq\{z:|z|<M_0\} \quad \text{for some $M_0\in\N$,}
\]
and so
\[
 \supp(C_\varphi(g))=\{z:\varphi(z)\in\supp(g)\}\subseteq\{z:|z|<N_0\} \quad \text{for some $N_0\in\N$.}
\]
In other words, $C_\varphi g$ has compact support.

We begin by taking $f_z$ as above, which is a continuous $L^p$-function on $\R$, for each $1\le p \le \infty$. We also take $g$ to be a continuous function on $\R$ with compact support. By the above, we have
\begin{align}
 \int_\R A_\varphi f_z(\alpha) g(\alpha) dm(\alpha) & = \int_\R P_{\Im (\varphi(z))} (\Re(\varphi(z)) - \alpha) g(\alpha) dm(\alpha) \notag\\
\tag{$\begin{smallmatrix} \text{since Poisson kernels are} \\ \text{reproducing kernels for $L^p$}\end{smallmatrix}$} & = g(\varphi(z)) \notag\\
& = C_\varphi g (z) \notag\\
& = \int_\R P_y(x-t) C_\varphi g(t) dm(t)  \label{2} \text{,}
\end{align}
since $C_\varphi g$ has compact support.

In order to continue, we will need the following.

\begin{thm}\label{boundedness}
Let $\varphi$ be an analytic self map of $\C^+$, which maps $\infty$ to itself. Then the operator, $A_\varphi$ is bounded on $L^p(\R)$ if and only if $C_\varphi$ is bounded on $L^q(\R)$, where $1/p+1/q=1$.
\end{thm}
\begin{proof}
 Suppose $C_\varphi$ is bounded on $L^q(\R)$. We begin by taking $f_z$ as above. Since $C_\varphi$ is bounded on $L^q$, it must map $L^q$ into itself. Moreover, $f_z\in L^p$ for each $p$, and so, by taking $L^q$ limits of the compact support function $g$ in \eqref{2}, we have that
\[
 \int_\R A_\varphi f_z(\alpha) g(\alpha) dm(\alpha) = \int_\R f_z(t) C_\varphi g(t) dm(t) \text{,}
\]
for all $g\in L^q(\R)$, since functions of compact support are dense in each $L^q$.

Taking suprema over all possible $g$ of norm $1$, we get (by the duality of $L^p$ and $L^q$)
\begin{align*}
 \| A_\varphi f_z \|_p & = \sup_{\|g\|=1} \int_\R f_z (t) C_\varphi g(t) dm(t) \\
& \le \|f_z\|_p \|C_\varphi\|_{L^q\rightarrow L^q} \text{,}
\end{align*}
and so $A_\varphi$ is bounded on Poisson kernels, and similarly, on finite linear combinations of Poisson kernels. We know, however, that the linear span of Poisson kernels is dense in each $L^p$, and hence by the Hahn-Banach theorem, $A_\varphi$ must be bounded on the whole of $L^p$.
\bigskip

Suppose now that $C_\varphi$ is not bounded on $L^q$. Then given $M\in\N$, there is a $g\in L^q$ with $\|g\|=1$ such that
\[
 \| C_\varphi g \| > M \text{.}
\]
As such, by the density of linear combinations of Poisson kernels in $L^p$, there must be some finite linear combination of Poisson kernels, $f$, with $\|f\|=1$ and
\begin{align*}
 \int_\R f(t) C_\varphi g(t) dm(t) & > M \text{.} \qquad
\intertext{Hence}
 \|A_\varphi f\|_p & > M \text{,} \qquad
\intertext{giving}
 \|A_\varphi\|& > M \text{,} \qquad
\end{align*}
and so $A_\varphi$ is not bounded.
\end{proof}

All this leads us to our first important result.

\begin{thm}\label{Thm1}
Assume $\varphi:\C^+\rightarrow\C^+$ is analytic, with $\varphi(\infty)=\infty$, and $1<p,q<\infty$ with $1/p+1/q=1$. Whenever $C_\varphi:L^q\rightarrow L^q$ is bounded, it is the adjoint of $A_\varphi:L^p\rightarrow L^p$.
\end{thm}
\begin{proof}
We return to the identity (2). Provided we ensure the integral remains finite, we may take linear combinations, and then limits of Poisson kernels, and similarly for continuous functions of compact support, and the same identity will clearly hold. Therefore, whenever both $A_\varphi$ and $C_\varphi$ are bounded, taking $L^p$ and $L^q$ limits respectively, since the Poisson kernels are dense in each $L^p$, and the continuous functions of compact support are dense in each $L^q$, the identity (2) remains true. In particular, $C_\varphi:L^q\rightarrow L^q$ is the adjoint of $A_\varphi:L^p\rightarrow L^p$ for each such $p$ and $q$.
\end{proof}

\subsection{More general analytic functions}We will now remove the assumption that $\varphi$ must map $\infty$ to itself. We first note the following, which is a slight extension of Corollary 2.2 from \cite{Matache}:

\begin{prop}\label{infmeasure}
 If $\varphi:\C^+\rightarrow\C^+$ is bounded on some set of infinite measure on $\R$, then $C_\varphi$ is not a bounded operator on $L^p(\R)$, or $H^p(\C^+)$ for any $1\le p <\infty$. It is also not bounded on $C_0(\R)$.
\end{prop}
\begin{proof}
 For $1\le p <\infty$, the function $f_p:\R\rightarrow\R$ given by
\[
 f_p(z) = \frac{1}{1+|z|^{2/p}}
\]
is in $L^p(\R)$. Moreover, each such function is in $C_0(\R)$.

If $\varphi$ is bounded on some set of infinite measure, say $\Sigma$, then we have
\[
 |\varphi(z)|<K
\]
on $\Sigma$, for some $K\in\N$. Now
\[
 \left|C_\varphi f_p(z)\right| = \left|\frac{1}{1+|\varphi(z)|^{2/p}} \right| > \frac{1}{1+K^{2/p}}
\]
on $\Sigma$. Since $\Sigma$ is of infinite measure, we have, setting $1/(1+K^{2/p})=\varepsilon_p$
\[
 m(\{z:C_\varphi f(z)>\varepsilon_p\}) = \infty \text{,}
\]
so $C_\varphi f_p \not\in L^p(\R)$, and hence $C_\varphi$ is not bounded on $L^p(\R)$. Since $C_\varphi f_p(z)>\varepsilon_p$ for arbitrarily large $z$, it is also clear that $C_\varphi f_p(z)\not\rightarrow0$ as $z\rightarrow\infty$, so $C_\varphi f_p\not\in C_0(\R)$, and hence $C_\varphi$ is not bounded on $C_0(\R)$. For the case of $H^p(\C^+)$, we take $g_p$ to be the function
\[
 g_p(z) = \frac{1}{(i+z)^{2/p}}\text{,}
\]
which is in $H^p(\C^+)$. The same argument will give that $C_\varphi g_p \not\in H^p(\C^+)$, indeed it will not even be in $L^p(\R)$. As such, $C_\varphi$ is not bounded on $H^p(\C^+)$.
\end{proof}

We know now that no function which is bounded on some set of infinite measure can give rise to a bounded composition operator. Let us suppose, therefore, that $\varphi$ is unbounded on every set of infinite measure, then for all $M\in\N$,
\[
 m(\{z:|\varphi(z)|<M\})<\infty \text{.}
\]
Indeed, for each $M\in\N$, given $\delta>0$, there exists an $N\in\N$ such that
\[
 m(K)<\delta\text{,}
\]
where
\[
 K=\{z:|\varphi(z)|<M, |z|>N\} \text{.}
\]
So, for all $M\in\N$, given $\varepsilon>0$, there is an $N_\varepsilon\in\N$ such that
\[
 m(\varphi(K_\varepsilon))<\varepsilon \text{,}
\]
where
\[
 K_\varepsilon=\{z:|\varphi(z)|<M, |z|>N_\varepsilon\} \text{.}
\]
Now, let $g$ have compact support, then there is some $M\in\N$ with
\[
 \supp(g)\subseteq\{z:|z|<M\} \text{.}
\]
Given $\varepsilon>0$, we can find an $N_\varepsilon\in\N$ such that
\[
 m(\varphi(K_\epsilon))<\epsilon \text{.}
\]
We now set $g_\ve=g\cdot\chi_{\R\setminus\vp(K_\ve)}$. If $|z|>N_\ve$, then either $\vp(z)\ge M$, in which case $g(\vp(z))=0$, or $z\in K_\ve$, in which case $\chi_{\R\setminus\vp(K_\ve)} (\vp(z))=0$. Either way,
\[
 C_\varphi g_\varepsilon(z) = g_\varepsilon \circ \varphi(z) = 0
\]
for $|z|>N_\varepsilon$, so $C_\varphi g_\varepsilon$ has compact support.

This motivates our next main result, which is a more general version of Theorem \ref{Thm1}.

\begin{thm}
Let $\varphi:\C^+\rightarrow\C^+$ be analytic, and let $1<p,q<\infty$ with $1/p+1/q=1$. Whenever $C_\varphi:L^q\rightarrow L^q$ is bounded, it is the adjoint of $A_\varphi:L^p\rightarrow L^p$.
\end{thm}

\begin{proof}
If $C_\vp$ is bounded, then $\vp$ must not be bounded on any set of infinite measure. We recall equation (2), taking now $g$ to be continuous with compact support, and $g_\varepsilon$ as above, we have
\begin{align*}
 \int_\R A_\varphi f_z(\alpha) g_\varepsilon(\alpha) dm(\alpha) & = \int_\R P_{\Im (\varphi(z))} (\Re(\varphi(z)) - \alpha) g_\varepsilon(\alpha) dm(\alpha) \\
\tag{$\begin{smallmatrix} \text{since Poisson kernels are} \\ \text{reproducing kernels for $L^P$}\end{smallmatrix}$} & = g_\varepsilon(\varphi(z)) \\
& = C_\varphi g_\varepsilon (z) \\
& = \int_\R P_y(x-t) C_\varphi g_\varepsilon(t) dm(t)  \tag{\ref{2}$'$} \text{,}
\end{align*}
which remains valid since $C_\varphi g_\varepsilon$ has compact support. We note that
\[
 \lim_{\varepsilon\rightarrow0} g_\varepsilon = g
\]
in each $L^p$-norm, so functions of this form are dense in the continuous functions of compact support, which are in turn dense in each $L^p$. Taking linear spans and closures, therefore, we have that $A_\varphi:L^p\rightarrow L^p$ is bounded if and only if $C_\varphi: L^q\rightarrow L^q$ is, and if both are bounded, then $C_\varphi$ is the adjoint of $A_\varphi$.
\end{proof}

Given that $L^2$ is a Hilbert space, we may then deduce the following corollary.

\begin{cor}
 If $C_\vp: L^2(\C^+)\rightarrow L^2(\C^+)$ is bounded, then $\A_\vp$ is its adjoint.
\end{cor}

If we replace the use of Poisson kernels in the preceeding results with the reproducing kernels for the $H^p$ spaces, namely the functions
\begin{equation}\label{repkernel}
 k_z(t) = \frac{1}{\overline{z}-t} \text{,}
\end{equation}
we obtain precisely the same results for the $H^p$ spaces. In particular, we have:

\begin{thm}
Let $1\le p<\infty$. If $C_\vp:H^p(\C^+)\rightarrow H^p(\C^+)$ is bounded, then it is the adjoint of $A_\vp:H^q\rightarrow H^q$, where $1/p+1/q=1$.
\end{thm}

\begin{cor}
 If $C_\vp: H^2(\C^+)\rightarrow H^2(\C^+)$ is bounded, then $\A_\vp:H^2(\C^+)\rightarrow H^2(\C^+)$ is its adjoint.
\end{cor}

\section{Rational self-maps of the upper half-plane}\label{Rational}
We use the mapping given in \cite{Chalendar03}, which identifies the Hardy space on the right half-plane, $H^p(\C_+)$, with the equivalent Hardy space on the disc, and the space $L^p(\T)$ with $L^p(i\R)$. We will need a slight alteration to work with the upper half-plane, $\C^+$, but this change is essentially trivial.

We begin by identifying the disc with the upper half-plane, via the mapping we have already mentioned, namely
\[
 \begin{array}{rlllcrl}
  J:&\D&\rightarrow&\C^+      & & z &\mapsto  i \left(\frac{1-z}{1+z}\right) \\
  J^{-1}:&\C^+&\rightarrow&\D & & s &\mapsto \frac{ i -s}{ i +s} \text{.}
 \end{array}
\]
This natural mapping then gives rise to a unitary equivalence between $H^p(\D)$, and $H^p(\C^+)$ ($1\le p<\infty$), given by
\[
 V:H^p(\D) \rightarrow H^p(\C^+)
\]
\begin{align*}
(Vg)(s) & = \frac{1}{\pi^{1/p}( i +s)^{2/p}} g\left(J^{-1}(s)\right)\\
(V^{-1}G)(z) & = \frac{(2 i )^{2/p}\pi^{1/p}}{(1+z)^{2/p}} G\left(J(z)\right) \text{,}
\end{align*}
the same mapping also identifies $L^p(\T)$ with $L^p(\R)$.

\begin{lem}
 If $\vp:\C^+\rightarrow\C^+$ is an analytic self-map of the upper half-plane, then the composition operator $C_\vp:H^p(\C^+)\rightarrow H^p(\C^+)$ (similarly $L^p(\R)\rightarrow L^p(\R)$) is unitarily equivalent to the weighted composition operator $L_\Phi:H^p(\D)\rightarrow H^p(\D)$ (similarly $L^p(\T)\rightarrow L^p(\T)$), given by
\[
 (L_\Phi f)(z) = \left(\frac{1+\Phi(z)}{1+z}\right)^{2/p} C_\Phi f(z) \text{,}
\]
where $\Phi=J^{-1}\circ\vp\circ J$.
\end{lem}
\begin{proof}
Let $f\in H^p(\C^+)$ (or $f\in L^p(\R)$), then
 \begin{align*}
  (V^{-1}\circ C_\vp \circ V f) & = \frac{(2 i )^{2/p}\pi^{1/p}}{(1+z)^{2/p}} (C_\vp \circ V f )(J(z)) \\
    & = \frac{(2 i )^{2/p}\cancel{\pi^{1/p}}}{(1+z)^{2/p}} \cdot \frac{1}{\cancel{\pi^{1/p}}( i +\vp(J(z)))^{2/p}} f(J^{-1}\circ \vp \circ J(z)) \text{.}
\intertext{Combining factors, we get}
& = \left(\frac{1}{1+z}\cdot\left(\cancelto{1}{\frac{ i  + \vp(J(z))}{ i  + \vp(J(z))}} +  \frac{i  - \vp(J(z))}{ i  + \vp(J(z))}\right)\right)^{2/p} f(\Phi(z))\\
    & = \left(\frac{1+\Phi(z)}{1+z}\right)^{2/p} C_\Phi f(z) \text{,}
\end{align*}
as required.
\end{proof}

We now recall Proposition \ref{infmeasure} above. For what follows, we will need the following corollary: 

\begin{cor}\label{a=>b}
 If $r:\C^+\rightarrow\C^+$ is a rational map such that $r(\infty)\neq\infty$, then $C_r$ is not bounded on $L^p(\R)$, or $H^p(\C^+)$ for any $1\le p <\infty$.
\end{cor}
\begin{proof}
 If $r(\infty)\neq\infty$, then $r$ must tend to some finite limit as $z\rightarrow \pm\infty$ (being rational). As such, there must be some $n\in\N$ such that $r$ is bounded on $\{z:|z|>n\}$, which has infinite measure, so by Proposition \ref{infmeasure}, $C_r$ is not bounded on any of the spaces mentioned.
\end{proof}

We now aim to prove that each rational map which does map $\infty$ to itself must give rise to a bounded operator on all the appropriate spaces.

\begin{prop}\label{b=>a}
 Let $r=a/b:\C^+\rightarrow\C^+$ be a rational map written in its lowest terms, and let $r(\infty)=\infty$. Then $C_r$ is bounded on each of the spaces, $H^p(\C^+)$, $L^p(\R)$ for $1\le p<\infty$.
\end{prop}
\begin{proof}
 We recall that
\begin{center}
$C_r$ is bounded on $H^p(\C^+), L^p(\R)$

$\Updownarrow$

$\left(\frac{1+\Phi_r}{1+z}\right)^{2/p} C_{\Phi_r}$ is bounded on $H^p(\D), L^p(\T)$

[where $\Phi_r = J^{-1} \circ r \circ J$]

$\Updownarrow$

$\sup_{\|f\|=1} \left\|\left(\frac{1+\Phi_r(z)}{1+z}\right)^{2/p} C_{\Phi_r} f \right\|_p \le \infty$
\end{center}
Now,
\begin{align*}
\sup_{\|f\|=1} \left\|\left(\frac{1+\Phi_r(z)}{1+z} \right)^{2/p} C_{\Phi_r} f \right\|_p & \le \sup_{\|f\|=1} \left\|\left(\frac{1+\Phi_r(z)}{1+z} \right)^{2/p} \right\|_\infty \cdot \left\|C_{\Phi_r} f\right\|_p \\
& = \left\|\left(\frac{1+\Phi_r(z)}{1+z} \right)^{2/p} \right\|_\infty \cdot \left\| C_{\Phi_r} \right\| \text{.}
\end{align*}
However $\|C_{\Phi_r}\|<\infty$ since all composition operators on the disc are bounded on the relevant spaces, so $C_r$ will be bounded on $H^p(\C^+)$ and $L^p(\R)$, provided
\[
 \left\| \left( \frac{1+J^{-1}\circ r \circ J(z)}{1+z} \right)^{2/p} \right\|_\infty<\infty \text{.}
\]
We note, however, that
\begin{align*}
 \left\| \left( \frac{1+J^{-1}\circ r \circ J(z)}{1+z} \right)^{2/p} \right\|_\infty & = \left\| \frac{1+J^{-1}\circ r \circ J(z)}{1+z} \right\|_\infty^{2/p} \\
 & = \left(\sup_{z\in\D} \left|\frac{1+J^{-1}\circ r \circ J(z)}{1+z} \right|\right)^{2/p} \text{,}
\end{align*}
so $C_r$ will be bounded on all the spaces simultaneously, provided
\[
 \sup_{z\in\D} \left|\frac{1+J^{-1}\circ r \circ J(z)}{1+z} \right| <\infty \text{.}
\]
Since $J^{-1}\circ r \circ J:\D\rightarrow\D$, we have $|J^{-1}\circ r \circ J(z)|<1$, so our inequality is clearly satisfied for $z$ away from $-1$. Hence
\[
 \sup_{z\in\D} \left|\frac{1+J^{-1}\circ r \circ J(z)}{1+z} \right| <\infty \Leftrightarrow \lim_{z\rightarrow-1} \left|\frac{1+J^{-1}\circ r \circ J(z)}{1+z} \right| <\infty \text{,}
\]
where here `$\lim$' denotes the non-tangential limit. Now, making the substitution $z=-k$,
\small
\[
 \lim_{z\rightarrow -1} \left|\frac{1+J^{-1}\circ r \circ J(z)}{1+z} \right| = \lim_{k\rightarrow 1}\left|\frac{\left(1+\frac{ i -r\left( i \frac{1+k}{1-k}\right)}{ i +r\left( i \frac{1+k}{1-k}\right)}\right)}{1-k}\right| = 2 \lim_{k\rightarrow 1}\left|\frac{1}{(1-k)( i +r\left( i \frac{1+k}{1-k}\right)} \right| \text{.}
\]
\normalsize
We recall that $r=a/b$, where $a$ and $b$ are polynomials with no common factors, so
\[
  2 \lim_{k\rightarrow 1}\left|\frac{1}{(1-k)\left( i +r\left( i \frac{1+k}{1-k}\right)\right)} \right| = 2 \lim_{k\rightarrow 1}\left|\frac{b\left( i \frac{1+k}{1-k}\right)}{(1-k)\left( i  b\left( i \frac{1+k}{1-k}\right)+a\left( i \frac{1+k}{1-k}\right)\right)} \right| \text{.}
\]
Making the change of variables $t=\frac{1}{1-k}$, that is $k=1+\frac{1}{t}$, we get
\[
 2 \lim_{k\rightarrow 1}\left|\frac{b\left( i \frac{1+k}{1-k}\right)}{(1-k)\left( i  b\left( i \frac{1+k}{1-k}\right)+a\left( i \frac{1+k}{1-k}\right)\right)} \right| = 2 \lim_{t\rightarrow \infty} \left| t \frac{b\left( -i (2t+1)\right)}{\left( i  b\left( i (2t+1)\right)+a\left( i (2t+1)\right)\right)} \right| \text{.}
\]
If we let $\deg(b)=m$, then the degree of the numerator of the fraction is $m+1$. Since $r(\infty)=\infty$, we must have $\deg(a)>\deg(b)$, so the degree of the denominator of the fraction is greater than or equal to $m+1$, so
\[
 \lim_{t\rightarrow \infty} \left| t \frac{b\left( -i (2t+1)\right)}{( i  b\left( i (2t+1)\right)+a\left( i (2t+1)\right)} \right| < \infty \text{,}
\]
and hence $C_r$ is bounded on each $L^p(\R)$, and each $H^p(\C^+)$.
\end{proof}
\begin{cor}\label{boundedrational}
 For a rational map $r:\C^+\rightarrow\C^+$, $C_r$ is bounded on each $L^p(\R)$, and each $H^p(\C^+)$ if and only if $r(\infty)=\infty$.
\end{cor}

\section{Further observations about rational maps}

\begin{prop}
 Let r be a rational map such that $r(\infty)=\infty$ and $r(\C^+)\subseteq \C^+$. Then both $r^{-1}(\C^+)$ and $r^{-1}(\C^-)$ contain an unbounded component.
\end{prop}
\begin{proof}
 The fact that $r^{-1}(\C^+)$ contains such a component is trivial, since $\C^+\subseteq r^{-1}(\C^+)$. For $r^{-1}(\C^-)$, we observe the following:

Let $h(z)=\frac{1}{z}=h^{-1}(z)$. Consider the mapping $hrh$, it is easy to see that $r(\infty)=\infty$ if and only if $hrh(0)=0$. Let $A_K$ be the region $\{z:|z|>K\}$. Since $r(\infty)=\infty$, $r(A_K)\subseteq A_K$ for sufficiently large $K$. Similarly, if $B_K=\{z:h(z)\in A_K\}$ then $hrh(B_K)\subseteq B_K$ for sufficiently large $K$. Moreover, $r(\C^+)\subseteq\C^+$, so $hrh(\C^-)\subseteq\C^-$.
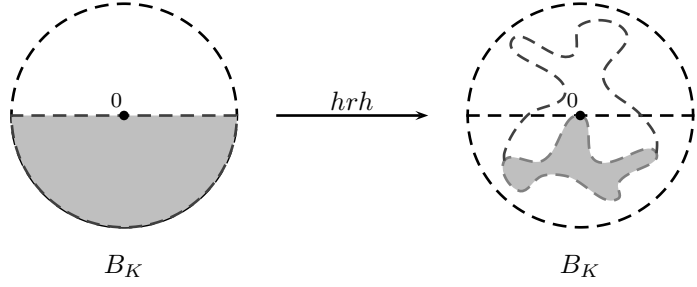
\begin{center}
\begin{pspicture}(12,5)
 \pscircle[linewidth=1pt,linestyle=dashed,fillstyle=solid,fillcolor=white](3,3){1.5}
 \pscircle[linewidth=1pt,linestyle=dashed,fillstyle=solid,fillcolor=white](9,3){1.5}
 \psline[linewidth=1pt]{->}(5,3)(7,3)
 \pswedge[linecolor=darkgray,linewidth=1pt,linestyle=dashed,fillcolor=lightgray,fillstyle=solid](3,3){1.5}{180}{360}
 \cput[linestyle=none](6,3.2){$hrh$}
 \psline[linecolor=black,linewidth=1pt,linestyle=dashed](7.5,3)(10.5,3)
 \psccurve[linecolor=gray,linewidth=1pt,linestyle=dashed,fillcolor=lightgray,fillstyle=solid](9,3)(9.2,2.4)(10,2.5)(9.5,2.1)(9.5,1.9)(8.8,2.2)(8.5,2.2)(8.2,2)(8,2.4)(8.5,2.4)
 \pscurve[linecolor=darkgray,linewidth=1pt,linestyle=dashed](8,2.4)(8,2.5)(8.3,3)(8.8,3.3)(8.1,4)(8.9,3.8)(9,4.2)(9.6,4.1)(9.3,3.5)(9.8,3)(10,2.7)(10,2.5)
 \psdots[dotstyle=*](3,3)(9,3)
 \cput[linestyle=none](2.9,3.2){\Small{$0$}}
 \cput[linestyle=none](8.9,3.2){\Small{$0$}}
 \cput[linestyle=none](3,1){$B_K$}
 \cput[linestyle=none](9,1){$B_K$}
\end{pspicture}
\end{center}
Now $B_K$ is an open neighbourhood of $0$, and $hrh$ is an open mapping, with $hrh(0)=0$, so $hrh(B_K)$ is an open neighbourhood of $0$. As such, $hrh(B_K)\not\subseteq\C^-$, and there is at least one point in $B_K$ (indeed, an open subset of $B_K$) which is mapped to $\C^+$ by $hrh$.

Thus, there is an open subset of $A_K$ mapped to $\C^-$ by $r$, but this is true for all sufficiently large $K$, so there are points of arbitrarily large modulus sent to $\C^-$. Since $r$ is rational, $r^{-1}(\C^-)$ has at most finitely many components, so $r^{-1}(\C^-)$ must have an unbounded component.
\end{proof}
\begin{prop}\label{rationalproperties}
 Let $r$ be a rational map such that $r(\infty)=\infty$ and $r(\C^+)\subseteq\C^+$. If $r$ is of the form
\[
 r(z) = \frac{a_nz^n+\ldots+a_1z+a_0}{b_mz^m+\ldots+b_1z+b_0}
\]
with $a_n,b_m\neq0$, then
\begin{enumerate}
 \item[(i)] $n=m+1$,
 \item[(ii)] $\frac{a_n}{b_m}\in\R$, and in particular, $\frac{a_n}{b_m}>0$,
 \item[(iii)] $\Im(\frac{a_0}{b_0})\ge0$.
\end{enumerate}
\end{prop}
\begin{proof}\text{}

 (i) For $|z|$ large enough, $r(z)\approx\frac{a_n}{b_n}z^{n-m}$. Taking $\frac{a_n}{b_m}=ce^{i\gamma}$, we set
\[
 \theta = \frac{\frac{3\pi}{2}-\gamma}{n-m}\text{.}
\]
If $n-m\ge2$, then $\theta\in\C^+$, but for sufficiently large $k$,
\[
r(ke^{i\theta})\approx ck^{n-m}e^{\frac{3\pi i}{2}} \in \C^- \text{,}
\]
so $r(\C^+)\not\subseteq\C^+$, which is a contradiction.

(ii) Since $n=m+1$, we have $r(z)\approx \frac{a_n}{b_m}z$, for sufficiently large $z$. Suppose $\frac{a_n}{b_m}\not\in\R^+$. Then
\[
 \frac{a_n}{b_m}=ce^{i\gamma} \text{,}
\]
where $\gamma\neq 0 \text{ }(\text{mod }2\pi)$, that is, $\gamma\in(0,2\pi)$. We have that $r(ke^{i\theta})\approx cke^{i(\theta+\gamma)}$, so setting $\theta=\pi-\frac{\gamma}{2}$, we get
\[
 r(ke^{i\theta})\approx  cke^{i(\pi+\frac{\gamma}{2})} \in\C^- \text{,}
\]
but $ke^{i(\pi-\frac{\gamma}{2})}\in\C^+$, so $r(C^+)\not\subseteq\C^+$, which is a contradiction.

(iii) For $z$ sufficiently small, we have
\[
 r(z)\approx \frac{a_0}{b_0} \text{.}
\]
If $\Im(\frac{a_0}{b_0})<0$, then 
\[
r(ke^{\frac{i\pi}{2}})\approx\frac{a_0}{b_0} \in\C^-\text{,}
\]
for $k$ sufficiently small, but $ke^{\frac{i\pi}{2}}\in\C^+$, so $r(\C^+)\not\subseteq\C^+$, which is a contradiction.
\end{proof}

Altogether, this gives us a refinement of Corollary \ref{boundedrational}, namely:

\begin{cor}\label{boundedrational2}
 For a rational map $r:\C^+\rightarrow\C^+$, $C_r$ is bounded on each $L^p(\R)$, and each $H^p(\C^+)$ if and only if the degree of the numerator of $r$ is precisely $1$ larger than the degree of the denominator of $r$.
\end{cor}

\section{A note on maps which are quotients of linear combinations of powers of $z$}
A slightly larger class of function which are of interest is the following: we denote by $QLP(A)$ the collection of maps from $A$ to $A$ which are quotients of linear combinations of powers of $z$. That is, all those maps of the form
\[
 \vp(z) = \frac{\lambda_1 z^{a_1} + \lambda_2 z^{a_2} + \ldots + \lambda_m z^{a_m}}{\mu_1 z^{b_1} + \mu_2 z^{b_2} + \ldots + \mu_n z^{b_n}} \text{,}
\]
where each $a_i$ and each $b_j$ is a non-negative real number. We assume without loss of generality that the powers $a_i$, and $b_i$ are written in descending order. A number of the methods we have used so far to work with rational maps will also work for these functions, and we present the results for completeness.

We note that each map $\vp\in QLP(\C^+)$ has a well-defined (possibly infinite) limit as $|z|\rightarrow\infty$, so by the same argument used in Corollary \ref{a=>b}, for such a $C_\vp$ to be bounded, we must have
\[
 \lim_{|z|\rightarrow\infty} \vp (z) = \infty \text{,}
\]
that is to say we must have $a_1>b_1$. Indeed more than this, we have the following:

\begin{prop}
 If $\vp\in QLP(\C^+)$, given by
\[
\vp(z) = \frac{\lambda_1 z^{a_1} + \lambda_2 z^{a_2} + \ldots + \lambda_m z^{a_m}}{\mu_1 z^{b_1} + \mu_2 z^{b_2} + \ldots + \mu_n z^{b_n}}
\]
is such that $a_1-b_1<1$, then $\vp$ does not give rise to a bounded compostion operator on any $L^p(\R)$, or $H^p(\C^+)$ for any $1\le p<\infty$.
\end{prop}
\begin{proof}
 Let $1\le p<\infty$, and let $\varepsilon>0$. Then the function $f_{p,\varepsilon}$ given by
\[
 f_{p,\varepsilon}(z) = \frac{1}{1+|z|^{\frac{1+\varepsilon}{p}}}
\]
is in $L^p(\R)$.
Let us suppose that $\vp\in QLP(\C^+)$, with $a_1-b_1<1$. Then in particular, $a_1-b_1<\frac{1}{1+\varepsilon}$ for some $\varepsilon>0$. Now
\[
 \left|C_\vp f_{p,\varepsilon}(z)\right|  = \left|\frac{1}{1+\left|\frac{\lambda_1 z^{a_1} + \ldots + \lambda_m z^{a_m}}{\mu_1 z^{b_1} + \ldots + \mu_n z^{b_n}}\right|^\frac{1+\varepsilon}{p}}\right| \ge \left|\frac{1}{1+\left(\frac{|\lambda_1|}{|\mu_1|}|z|^\frac{1}{\cancel{1+\varepsilon}}\right)^\frac{\cancel{1+\varepsilon}}{p}}\right| \text{,}
\]
for sufficiently large $z$. This is clearly not an $L^p(\R)$ function, and so $C_\vp$ is not bounded on $L^p(\R)$. We note again, that much as with rational functions, the map
\[
 f(z) = \frac{1}{(i+z)^{\frac{1+\ve}{p}}}
\]
will do for the $H^p$ case.
\end{proof}

So if $a_1-b_1<1$, $C_\vp$ cannot be bounded. It remains only to show that if $a_1-b_1\ge1$, then $C_\vp$ must be bounded.
\begin{prop}
 Let $\vp\in QLP(\C^+)$, with representation
\begin{equation}\label{QLP}
 \vp(z) = \frac{\lambda_1 z^{a_1} + \lambda_2 z^{a_2} + \ldots + \lambda_m z^{a_m}}{\mu_1 z^{b_1} + \mu_2 z^{b_2} + \ldots + \mu_n z^{b_n}} \text{,} 
\end{equation}
moreover, let $a_1-b_1\ge 1$. Then $\vp$ gives rise to a bounded composition operator on each of the spaces $H^p(\C^+)$, $L^p(\R)$ for $1\le p<\infty$.
\end{prop}
\begin{proof}
 We begin by writing
\begin{align*}
 \sigma(z) & = \lambda_1 z^{a_1} + \lambda_2 z^{a_2} + \ldots + \lambda_m z^{a_m} \\
 \tau(z)   & = \mu_1 z^{b_1} + \mu_2 z^{b_2} + \ldots + \mu_n z^{b_n} \text{,}
\end{align*}
then $\vp=\sigma/\tau$. Using the same argument as in Proposition \ref{b=>a}, we get that $C_\vp$ is bounded, provided
\[
 \lim_{t\rightarrow\infty} \left| t \frac{\tau(-i(2t+1))}{i\sigma(-i(2t+1))+\tau(-i(2t+1))}\right| <\infty \text{.}
\]
The leading power in the numerator of the fraction is $b_1+1$, and in the denominator, it is $a_1$, but $a_1-b_1\ge 1$, so this limit is indeed finite.
\end{proof}
\begin{cor}
 A map of the form \eqref{QLP} induces a bounded composition operator if and only if $a_1-b_1\ge 1$.
\end{cor}

\section{An adjoint formula for rational-symbol composition operators}\label{Residue}

We begin by making some elementary calculations concerning $C_\vp^*$. Let $\vp$ be a rational self-map of $\C^+$ with $\vp(\infty)=\infty$, and let $f\in H^2(\C^+)$. If we denote by $k_z$ the reproducing kernel for $H^2$ at $z$ as defined in \eqref{repkernel}, then
\begin{align}
 (C_\vp^*f)(z) & = \left<C_\vp^*f,k_z\right> \notag\\
               & = \left<f,C_\vp k_z\right> \notag\\
               & = \int_\R f(t)\cdot \overline{\frac{1}{2\pi i}}\cdot\overline{\frac{1}{\overline{z}-\vp(t)}} dt \notag\\
               & = \frac{1}{2\pi i}\int_\R \frac{f(t)}{\overline{\vp(t)}-z} dt \label{3} \text{.}
\end{align}
Now let us consider the closed curve $\gamma_\ve$ in $\C^+$ shown below:
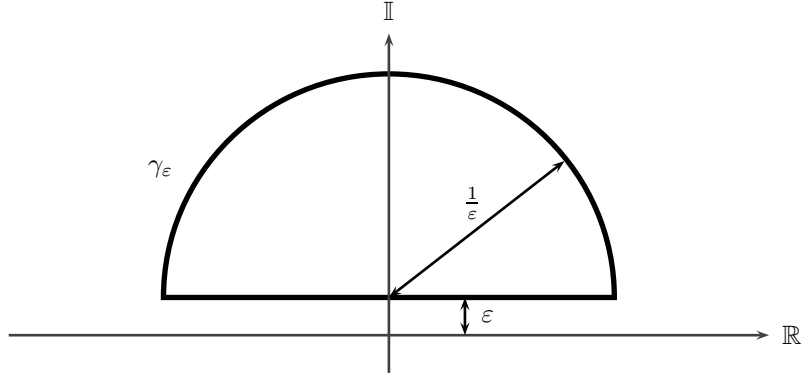
\begin{center}
\begin{pspicture}(12,6)
 \psline[linewidth=1pt,linecolor=darkgray]{->}(1,1)(11,1)
 \pswedge[linewidth=2pt](6,1.5){3}{0}{180}
 \cput[linestyle=none](3,3.2){$\gamma_\ve$}
 \psline[linewidth=1pt,linecolor=darkgray]{->}(6,0.5)(6,5)
 \cput[linestyle=none](11.3,1){$\R$}
 \cput[linestyle=none](6,5.3){$\mathbb{I}$}
 \psline[linewidth=1pt]{<->}(7,1)(7,1.5)
 \cput[linestyle=none](7.3,1.25){$\ve$}
 \psline[linewidth=1pt]{<->}(6,1.5)(8.31,3.31)
 \cput[linestyle=none](7.1,2.75){$\frac{1}{\ve}$}
\end{pspicture}
\end{center}

We note that
\[
 \frac{1}{2\pi i} \oint_{\gamma_\ve} \frac{f(t)}{\overline{\vp(\overline{t})}-z} dt = \frac{1}{2\pi i} \int_{(-\frac{1}{\ve},\frac{1}{\ve})} \frac{f(t+\ve i)}{\overline{\vp(\overline{t+\ve i})}-z} dt + \frac{1}{2\pi i} \int_{\kappa_\epsilon} \frac{f(t)}{\overline{\vp(\overline{t})}-z} dt \text{,}
\]
where ${\kappa_\epsilon}$ denotes the semicircular section of $\gamma_\ve$. Taking limits as $\ve\rightarrow0$, we get
\small
\[
 \lim_{\ve\rightarrow0} \frac{1}{2\pi i} \oint_{\gamma_\ve} \frac{f(t)}{\overline{\vp(\overline{t})}-z} dt = \lim_{\ve\rightarrow0} \frac{1}{2\pi i} \int_{(-\frac{1}{\ve},\frac{1}{\ve})} \frac{f(t+\ve i)}{\overline{\vp(\overline{t+\ve i})}-z} dt + \frac{1}{2\pi i}\lim_{\ve\rightarrow0} \int_{\kappa_\epsilon} \frac{f(t+\ve i)}{\overline{\vp(\overline{t+\ve i})}-z} dt \text{.}
\]
\normalsize
Let us now consider the collection of functions $f$ in $H^2$, such that $f=O(z^{-1})$ near $\infty$. For $f$ in this collection,
\[
 \lim_{\ve\rightarrow0} \int_{\kappa_\epsilon} \frac{f(t)}{\overline{\vp(\overline{t})}-z} dt = 0
\]
since $\vp(\infty)=\infty$. As such,
\begin{equation}
  \lim_{\ve\rightarrow0} \oint_{\gamma_\ve} \frac{f(t)}{\overline{\vp(\overline{t})}-z} dt = \lim_{\ve\rightarrow0} \int_{(-\frac{1}{\ve},\frac{1}{\ve})} \frac{f(t+\ve i)}{\overline{\vp(\overline{t+\ve i})}-z} dt = \int_\R \frac{f(t)}{\overline{\vp(t)}-z} dt \label{4}\text{,}
\end{equation}
since $t=\overline{t}$ for $t\in\R$. Combining \eqref{3} and \eqref{4}, we get that
\[
 (C_\vp^*f)(z) = \lim_{\ve\rightarrow0} \frac{1}{2\pi i} \oint_{\gamma_\ve} \frac{f(t)}{\overline{\vp(\overline{t})}-z} dt  = \sum_{\stackrel{t\in\C^+}{\overline{\varphi(\overline{t})}=z}} \Res\left(\frac{f(s)}{\overline{\vp(\overline{s})}-z},s=t\right) \text{,}
\]
by The Residue Theorem. Since the collection of functions which are $O(z^{-1})$ near $\infty$ are dense in $H^2$, we can write any function $f$ in $H^2$ as
\[
 f = \lim_{n\rightarrow\infty} f_n \text{,}
\]
where the $f_n$ are $O(z^{-1})$ near $\infty$. As such,
\begin{align*}
  \lim_{\ve\rightarrow0} \frac{1}{2\pi i} \oint_{\gamma_\ve} \frac{f(t)}{\overline{\vp(\overline{t})}-z} dt \hspace{-3cm} \\
 & = \lim_{\ve\rightarrow0} \frac{1}{2\pi i} \int_{(-\frac{1}{\ve},\frac{1}{\ve})} \frac{f(t+\ve i)}{\overline{\vp(\overline{t+\ve i})}-z} dt + \lim_{n\rightarrow\infty} \cancelto{0}{\lim_{\ve\rightarrow0} \int_{\kappa_\epsilon} \frac{f_n(t)}{\overline{\vp(\overline{t})}-z} dt} \text{,}
\end{align*}
and the same result carries through. This gives us a formula for $C_\vp^*$, namely
\[
 (C_\vp^*f)(z) = \sum_{t\in\C^+\cap\overline{\vp^{-1}(\overline{z})}} \Res\left(\frac{f(s)}{\overline{\vp(\overline{s})}-z},s=t\right) \text{.}
\]
We note further that, if we assume that $\frac{1}{\overline{\vp(\overline{s})}-z}$ has only simple poles, then
\begin{align*}
 (C_\vp^*f)(z) & = \sum_{t\in\C^+\cap\overline{\vp^{-1}(\overline{z})}} \Res\left(\frac{f(s)}{\overline{\vp(\overline{s})}-z},s=t\right) = \sum_{t\in\C^+\cap\overline{\vp^{-1}(\overline{z})}} \lim_{s\rightarrow t} \frac{(s-t)f(s)}{\overline{\vp(\overline{s})}-z} \\
  & = \sum_{t\in\C^+\cap\overline{\vp^{-1}(\overline{z})}} \left(\lim_{s\rightarrow t} \frac{s-t}{\overline{\vp(\overline{s})}-z}\right) \left(\lim_{s\rightarrow t} f(s)\right)  \text{,}
\end{align*}
the last line being possible because $t$ is only a simple pole, and $f$ has no poles, being analytic. Overall, this gives us
\begin{align}
 (C_\vp^*f)(z) & = \sum_{t\in\C^+\cap\overline{\vp^{-1}(\overline{z})}} \left(\lim_{s\rightarrow t} \frac{s-t}{\overline{\vp(\overline{s})}-z}\right)f(t) \label{CphiResidue}\text{,}
\end{align}
which means that $C_\vp^*$ is in fact a so-called `multiple-valued weighted composition operator'. We note finally that, since $\vp$ is rational, it will have only simple poles for all but at most finitely many $z$, hence the above formula is valid except for possibly finitely many $z$, that is to say, it is true on a dense subset of $\C^+$.

\section{Some Examples}

Using the formulae we have derived, we will calculate the adjoints of a number of composition operators. First though, in order to use our Aleksandrov Operator characterisation, we will need to work out how to calculate the AC measures associated with an analytic function $\vp$. The following are the equivalent of a number of useful results on the disc from Saksman's excellent introduction to AC measures \cite{Saksman}.

Let $\mu=\mu^a dm +d\sigma$ be a measure on $\R$, and let us also denote its Poisson extension by $\mu$, that is to say
\[
 \mu(z) = \int_\R P_z(\zeta) d\mu(\zeta)\text{.}
\]
From Theorem 11.24, and a simple extension of Exercise 19, Section 11 in \cite{Rudin}, we have
\begin{equation}\label{eqn1}
\lim_{r\rightarrow 0^+} \mu(b+ i  r) = \begin{cases}
                                           \mu^a & \text{for $m$-almost every $b\in\R$,} \\
                                           \infty & \text{for $\sigma$-almost every $b\in\R$.}
                                          \end{cases}
\end{equation}
Given an analytic self-map, $\varphi$ of $\C^+$, we recall that the Aleksandrov-Clark (AC) measures, $(\mu_\alpha,c_\alpha)$, of $\varphi$ are defined by the formula
\[
 \frac{1}{\pi(1+\alpha^2)} \Re \left(\frac{ i (1+\alpha\varphi(z))}{\varphi(z)-\alpha}\right) = \int_\R P_y (x-t) d\mu_\alpha(t) +c_\alpha y \text{,}
\]
where $z=x+ i  y$.

\begin{prop}\label{ACmeasure1}
 If $\varphi:\C^+\rightarrow\C^+$ is an analytic function, and $\{\mu_\alpha\}$ is its collection of AC measures, then
\[
 \mu_\alpha^a(\zeta) = \begin{cases}
                       \frac{1}{\pi(1+\alpha^2)} \Re \left(\frac{ i (1+\alpha\varphi(\zeta))}{\varphi(\zeta)-\alpha}\right) & \text{if $\varphi(\zeta)\in\C\setminus\R$,} \\
                       0 & \text{if $\varphi(\zeta)\in\R$.}
                      \end{cases}
\]
\end{prop}
\begin{proof}
\[
 \int_\R P_y(x-t)d\mu_\alpha(t) +c_\alpha y = \frac{1}{\pi(1+\alpha^2)} \Re\left(\frac{ i (1+\alpha\varphi(z))}{\varphi(z)-\alpha}\right) \text{,}
\]
where
\[
 P_y(x-t) = \frac{1}{\pi} \frac{y}{(x-t)^2 + y} \text{.}
\]
Hence,
\begin{equation}\label{eqn2}
 \underbrace{\frac{1}{\pi}\int_\R \frac{y}{(x-t)^2 + y^2} d\mu_\alpha(t)}_{=\mu_\alpha(x+ i  y)} = \frac{1}{\pi(1+\alpha^2)}\Re\left(\frac{ i (1+\alpha\varphi(z))}{\varphi(z)-\alpha}\right) - c_\alpha y \text{,}
\end{equation}
recalling that $\mu_\alpha$ denotes both a measure and its Poisson extension. We  now take limits as $y\rightarrow0$. The left hand side of \eqref{eqn2} is $\mu_\alpha^a(x)$ ($m$-a.e.) by \eqref{eqn1}, and the right hand side is
\[
 \frac{1}{\pi(1+\alpha^2)}\Re\left(\frac{ i (1+\alpha\varphi(x))}{\varphi(x)-\alpha}\right) =
\begin{cases}
\frac{1}{\pi(1+\alpha^2)}\Re\left(\frac{ i (1+\alpha\varphi(x))}{\varphi(x)-\alpha}\right) & \text{if $\varphi(x)\in\C\setminus\R$,} \\
0 & \text{if $\varphi(x)\in\R$.}
\end{cases}
\]
We note that $m$-almost everywhere equality is the best we could hope for, given that $\mu_\alpha^a$ is an $L^1$ function.
\end{proof}

\begin{prop2}\label{ACmeasure2}
 \[
  \supp(\sigma_\alpha) \subseteq \{x\in\R:\varphi(x)=\alpha\} \text{.}
 \]
\end{prop2}
\begin{proof}
 Suppose $x\in\R$, such that either
\[
 \lim_{y\rightarrow0^+}\vp(x+ i  y) \neq 0 \text{,}
\]
or this limit does not exist. Then there exists some $\ve>0$ and some sequence $y_n\searrow0$ with
\[
 |\vp(x+ i  y_n)| \ge \ve \text{,}
\]
for each $n\in\N$. But
\[
 \mu_\alpha(x+ i  y_n) = \frac{1}{\pi(1+\alpha^2)}\Re\left(\frac{ i (1+\alpha\varphi(x+ i  y))}{\varphi(x+ i  y)-\alpha}\right) - c_\alpha y \text{.}
\]
Since $\vp(x)\neq0$, we have
\[
 \liminf_{y\rightarrow 0+}\mu_\alpha(x+ i  y) < \infty \text{,}
\]
so by \eqref{eqn1}, $\sigma({x})=0$. As such,
\[
 \supp(\sigma_\alpha) \subseteq \{x\in\R:\varphi(x)=\alpha\} \text{.}
\]
\end{proof}

We now move on to our first example, which is the simplest possible case of a composition operator with linear symbol. It is worth noting that by results from section \ref{Rational}, strictly linear maps are the only fractional linear maps which yield bounded composition operators. Moreover, in order to map $\C^+$ into itself, we must have positive real co-efficient of $z$, and a constant term with non-negative imaginary part, by Proposition \ref{rationalproperties}.

\begin{eg}\label{eg1}
We begin by noting that, if $\vp(z)=az+b$, where $a\in\R^+$, and $\Im(b)\ge0$, then
\begin{align*}
 \left<C_\vp f,g\right> & = \int_\R C_\vp f (z) \overline{g(z)} dz = \int_\R f(az+b) \overline{g(z)} dz \text{,}
\intertext{setting $x=az+b$,}
              & = \int_\R f(x) \overline{g\left(\frac{x-b}{a}\right)} \frac{dz}{dx} dx = \int_\R f(x) \overline{\frac{1}{a}g\left(\frac{x-\overline{b}}{a}\right)} dx \text{,}
\end{align*}
since the analytic extension of $g$ to the lower half-plane is $g(\overline{z})$. So the adjoint of $C_\vp$ is the weighted composition operator given by
\begin{equation}\label{Cphieg1}
 C_\vp^* f (z) = \frac{1}{a} f\left(\frac{z-\overline{b}}{a}\right) \text{.}
\end{equation}

The calculation of the same adjoint using AC measures is as follows.

We must split the example into two cases:
\begin{enumerate}
 \item[(i)] the case where $\Im(b)=0$.
 \item[(ii)] the case where $\Im(b)>0$.
\end{enumerate}

\subsection*{(i)} Since $\Im(b)=0$, we have that $\vp(x)\in\R$ for all $x\in\R$, so by Proposition \ref{ACmeasure1}, the absolutely continuous part of each AC measure associated with $\vp$ is identically $0$, or in other words, each measure is entirely singular. By Proposition \ref{ACmeasure2}, the singular part of each $\mu_\alpha$ lives on the preimages of $\alpha$ under $\vp$, so the support of each $\mu_\alpha$ is just the single point $\frac{\alpha-b}{a}$.

In order to determine the value of the point mass at $\frac{\alpha-b}{a}$, we use the defining equation for the AC measures of $\vp$, namely
\[
c_\alpha y + \int_\R P_y(x-t) d\mu_\alpha(t) = \frac{1}{\pi(1+\alpha^2)}\Re \left(\frac{i (1 + \alpha\varphi(x+ i y))}{\varphi(x + i y) - \alpha} \right)\text{,}
\]
where $P_y(x-t)=\frac{1}{\pi}\frac{y}{(x-t)^2+y^2}$. Since $\infty$ has no preimages under $\vp$, $c_\alpha=0$ for each $\alpha$. Setting $x=0$ and $y=1$ in the above, we get
\[
 \cancel{\frac{1}{\pi}}\int_\R \frac{1}{1+t^2} d\mu_\alpha(t) = \cancel{\frac{1}{\pi}}\cdot\frac{1}{1+\alpha^2}\Re \left(\frac{i (1 + \alpha (ai+b))}{(ai + b) - \alpha} \right) \text{.}
\]
As such,
\[
 \frac{1}{1+\left(\frac{\alpha-b}{a}\right)^2}\mu_\alpha\left(\left\{\frac{\alpha-b}{a}\right\}\right)  = \frac{1}{1+\alpha^2}\Re\left(\frac{i(1+\alpha(ai+b))}{ai+b-\alpha}\right)
  = \frac{a}{(b-\alpha)^2+a^2}\text{,}
\]
so
\[
 \mu_\alpha\left(\left\{\frac{\alpha-b}{a}\right\}\right) = \frac{1}{a} \text{.}
\]
Now
\[
 C_\vp^* f (\alpha) = A_\vp f (\alpha) = \int_\R f(t) d\mu_\alpha(t) = \frac{1}{a}f\left(\frac{\alpha-b}{a}\right) \text{,}
\]
which is precisely the same as \eqref{Cphieg1}, since $b\in\R$.

\subsection*{(ii)} Since $\Im(b)>0$, Proposition \ref{ACmeasure2} tells us that each $\mu_\alpha$ is absolutely continuous and Proposition \ref{ACmeasure1} that
\begin{align*}
 \mu_\alpha^a (t) & = \frac{1}{\pi(1+\alpha^2)}\Re\left(\frac{ i (1+\alpha\varphi(t))}{\varphi(t)-\alpha}\right) \\
 & = \frac{1}{\pi(1+\alpha^2)}\Re\left(\frac{ i (1+\alpha(at+b)}{at+b-\alpha}\right)\text{.}
\end{align*}
So
\begin{align*}
 (C_\vp^*f)(\alpha) & = \int_\R f(t) \frac{1}{\pi(1+\alpha^2)}\Re\left(\frac{ i (1+\alpha(at+b)}{at+b-\alpha}\right) dt \\
& = \int_\R f(t) \frac{1}{\pi} \frac{\Im(b)}{\underbrace{(at-\alpha+\Re(b))^2}_{=(\alpha-\Re(b)-at)^2}+\Im(b)^2}dt \text{.}
\end{align*}
Since $\alpha\in\R$, $\Re(\alpha)=\alpha$ and $\Im(\alpha)=0$, so
\begin{align*}
 (C_\vp^*f)(\alpha) & = \int_\R f(t) \frac{1}{\pi} \frac{1}{a}\frac{\Im\left(\frac{\alpha-\overline{b}}{a}\right)}{\left(\Re\left(\frac{\alpha-\overline{b}}{a}\right)-t\right)^2+\Im\left(\frac{\alpha-\overline{b}}{a}\right)^2}dt \\
& = \frac{1}{a}\int_\R f(t) P_{\left(\frac{\alpha-\overline{b}}{a}\right)} (t) dt \\
& = \frac{1}{a}f\left(\frac{\alpha-\overline{b}}{a}\right)\text{,}
\end{align*}
just as in \eqref{Cphieg1}.
\bigskip

Our final method of calculation is the residue formula from section \ref{Residue}. We note that since $\vp$ is linear, it has a well defined inverse, and no repeated roots, so in fact formula \eqref{CphiResidue} describes the adjoint of $\vp$ everywhere. As such, we have
\begin{align*}
(C_\vp^*f)(z) & = \sum_{t\in\C^+\cap\overline{\vp^{-1}(\overline{z})}} \left(\lim_{s\rightarrow t} \frac{s-t}{\overline{\vp(\overline{s})}-z}\right)f(t) \\
 & = \left(\lim_{s\rightarrow \frac{z-\overline{b}}{a}} \frac{s-\frac{z-\overline{b}}{a}}{as+\overline{b}-z}\right)f\left(\frac{z-\overline{b}}{a}\right) \text{.}
\end{align*}
So a simple application of L'H\^opital's Theorem gives us
\[
 (C_\vp^* f)(z) = \frac{1}{a}f\left(\frac{z-\overline{b}}{a}\right) \text{,}
\]
just as in \eqref{Cphieg1}.
\end{eg}

\begin{eg}\label{eg2}
 Let us consider the map
\[
 \vp(z) = z-\frac{1}{z} \text{,}
\]
which we know to give rise to an isometric composition operator on $H^2(\C^+)$ by Proposition 2.1 of \cite{Chalendar03}. We observe that
\[
 \vp^{-1}(z) = \left\{\frac{z\pm\sqrt{z^2+4}}{2}\right\} \text{.}
\]

For calculating the AC measures of $\vp$, we note that, if $d\mu_\alpha=\mu_\alpha^a dm+d\sigma_\alpha$, then:
\begin{enumerate}
 \item[(a)] $z-\frac{1}{z}\in\R$ for all $z\neq0$, so $\mu_\alpha^a(x)=0$ for all $x$,
 \item[(b)] $\sigma_\alpha$ lives on $\{x:\vp(x)=\alpha\}=\{x:x-\frac{1}{x}=\alpha\}=\{\frac{\alpha\pm\sqrt{\alpha^2+4}}{2}\}$.
\end{enumerate}
Moreover, $\mu_0\equiv0$. Setting $x=0$ and $y=1$, we get
\begin{align}
 \frac{1}{\pi}\int_\R \frac{1}{1+t^2}d\mu_\alpha(t) & = \frac{1}{\pi(1+\alpha^2)} \Re\left(\frac{ i (1+\alpha\varphi( i ))}{\varphi( i )-\alpha}\right)  - c_\alpha \text{,} \label{sim1}
\intertext{and setting $x=0$ and $y=2$, we get}
\frac{1}{\pi}\int_\R \frac{1}{4+t^2}d\mu_\alpha(t) & = \frac{1}{\pi(1+\alpha^2)} \Re\left(\frac{ i (1+\alpha\varphi(2 i ))}{\varphi(2 i )-\alpha}\right)  - 2c_\alpha \text{.}\label{sim2}
\end{align}
It's easy to show that $c_\alpha=0$ for all $\alpha$, so solving \eqref{sim1} and \eqref{sim2} as simultaneous equations gives us
\begin{align*}
\sigma_\alpha\left(\left\{\frac{\alpha+\sqrt{\alpha^2+4}}{2}\right\}\right) &= \frac{\sqrt{\alpha^2+4}+\alpha}{2\sqrt{\alpha^2+4}}\quad \intertext{and} \quad
\sigma_\alpha\left(\left\{\frac{\alpha-\sqrt{\alpha^2+4}}{2}\right\}\right) &= \frac{\sqrt{\alpha^2+4}-\alpha}{2\sqrt{\alpha^2+4}}\text{.}
\end{align*}

So
\[
 (C_\vp^*f)(\alpha) = \frac{\sqrt{\alpha^2+4}+\alpha}{2\sqrt{\alpha^2+4}}f\left(\frac{\alpha+\sqrt{\alpha^2+4}}{2}\right) + \frac{\sqrt{\alpha^2+4}-\alpha}{2\sqrt{\alpha^2+4}}f\left(\frac{\alpha-\sqrt{\alpha^2+4}}{2}\right) \text{.}
\]
We note as an aside that
\[
 \|\mu_\alpha\|=\sigma_\alpha^+ + \sigma_\alpha^- = \frac{\sqrt{\alpha^2+4}+\cancel{\alpha}+\sqrt{\alpha^2+4}-\cancel{\alpha}}{2\sqrt{\alpha^2+4}} = 1 \text{,}
\]
for each $\alpha$.
\bigskip

The residue method calculation proceeds as follows: let us suppose that $z\neq2i$, that is to say that the two values of $\vp^{-1}(z)$ are distict. Then
\[
 (C_\vp^*f)(z) = \sum_{t\in\C^+\cap\overline{\vp^{-1}(\overline{z})}} \left(\lim_{s\rightarrow t} \frac{s-t}{\overline{\vp(\overline{s})}-z}\right)f(t)\text{.}
\]
So, using L'H\^opital's Theorem to evaluate the limit, we get
\[
 (C_\vp^*f)(z) = \frac{\sqrt{z^2+4}+z}{2\sqrt{z^2+4}}f\left(\frac{z+\sqrt{z^2+4}}{2}\right) + \frac{\sqrt{z^2+4}-z}{2\sqrt{z^2+4}}f\left(\frac{z-\sqrt{z^2+4}}{2}\right)\text{,}
\]
for $z\neq2i$.

Moreover, we observe that for $z=2i$, we have only one solution to $\overline{\vp(\overline{t})}=z$, namely $t=i$. In this case,
\[
 \Res\left(\frac{f(s)}{\overline{\vp(\overline{s})}-2i},s=i\right) = f(i)\text{,}
\]
so since
\[
 \frac{\sqrt{z^2+4}+z}{2\sqrt{z^2+4}} + \frac{\sqrt{z^2+4}-z}{2\sqrt{z^2+4}} = 1 \text{,}
\]
the formula is still valid for $z=2i$.
\end{eg}

By observing the above examples, we have reason to hope that some other well known results from the disc may have natural analogues in the half-plane. In Example \ref{eg1} above, $C_\vp$ is an isometry when $b=0$, and $|a|=1$, and moreover, we have already observed that Example \ref{eg2} gives an isometry. In both these cases, $\vp$ is inner (it maps the boundary of $\C^+$ to itself), and the AC measures associated with $\vp$ are of constant modulus equal to $1$. We see in \cite{Saksman}, for example, that this is precisely the condition for a composition operator on the disc to give rise to an isometry.

We note finally, that by virtue of the mapping at the beginning of Section \ref{Rational}, every composition operator on $\C^+$ is equivalent to a weighted composition operator on the disc. As such, our observations here may be used to study a certain class of weighted composition operators, which are also of interest.
\bibliographystyle{siam}

\end{document}